\documentclass[a4paper,11pt]{article}

\usepackage{amsmath}
\usepackage{amsthm}
\usepackage{amsfonts}
\usepackage{amssymb}
\usepackage{amsrefs}
\usepackage{enumerate}
\usepackage{xfrac}
\usepackage{color} % just for comments can be deleted at the end

\usepackage{mathpazo}     % Change the font style to Palantino
\fontsize{11}{14}  % Trying to change the leading to make the text more legible

\usepackage[usenames,dvipsnames]{xcolor}
\usepackage{hyperref}
\hypersetup{
   colorlinks,
   citecolor=ForestGreen,
   filecolor=blue,
   linkcolor=Maroon,
   urlcolor=orange
}

\newtheorem{theorem}{Theorem}[section]
\newtheorem{proposition}[theorem]{Proposition}
\newtheorem{lemma}[theorem]{Lemma}
\newtheorem{corollary}[theorem]{Corollary}
\theoremstyle{definition}
\newtheorem{definition}[theorem]{Definition}

\newtheorem{remark}[theorem]{Remark}
\newtheorem{convention}[theorem]{Convention}

\newcommand{\NN}{\mathbb{N}}
\newcommand{\RR}{\mathbb{R}}
\newcommand{\ZZ}{\mathbb{Z}}
\newcommand{\GLnZ}{GL_{n}(\ZZ)}
\newcommand{\GLnFt}{GL_{n}(F[t])}
\DeclareMathOperator{\diam}{diam}

\newcommand{\Space}{X}

\newcommand{\dimn}{n}

\newcommand{\Group}{G}
\newcommand{\groupelt}{g}

\newcommand{\Family}{\mathcal{F}}
\newcommand{\Vcyc}{\mathcal{V}cyc}
\newcommand{\GroupFin}{F}
\newcommand{\groupfinelt}{g}
\newcommand{\groupfineltprime}{g'}
\newcommand{\groupfineltother}{h}
\newcommand{\groupfineltotherprime}{h'}

\newcommand{\FMetricSpace}{Z}
\newcommand{\OpenSubsetsZ}{\mathcal{P}^{\circ}(\FMetricSpace)}
\newcommand{\fmetricspaceelt}{z}
\newcommand{\fmetricspaceeltprime}{z'}
\newcommand{\FMetricSpaceSubset}{Y}
\newcommand{\FMetricSpaceQuotient}{\GroupFin \backslash \FMetricSpace}

\newcommand{\GSpace}{Y}
\newcommand{\GMetricSpace}{X}
\newcommand{\GMetricSpaceCpct}{\overline{X}}
\newcommand{\CoverArb}{\mathcal{U}}
\newcommand{\coverelt}{U}
\newcommand{\RefQuotient}{\mathcal{V}}
\newcommand{\refquotelt}{V}
\newcommand{\refquoteltsuperset}{\coverelt_{\refquotelt}}
\newcommand{\refquoteltprime}{\refquotelt'}
\newcommand{\refquoteltprimesuperset}{\coverelt_{\refquoteltprime}}
\newcommand{\Refinement}{\mathcal{W}}
\newcommand{\refelt}{W}
\newcommand{\refeltexpression}{\pi^{-1}(\refquotelt) \cap \groupfineltother \cdotp \refquoteltsuperset}
\newcommand{\refeltprime}{W'}
\newcommand{\refeltprimeexpression}{\pi^{-1}(\refquoteltprime) \cap \groupfineltotherprime \cdotp \refquoteltprimesuperset}

\newcommand{\Collection}{\mathcal{U}}
\newcommand{\collectionelt}{U}
\newcommand{\BoundaryBound}{k}

\newcommand{\IndCover}{\mathcal{V}}
\newcommand{\indcoverelt}{V}
\newcommand{\indcovereltclosure}{\overline{\indcoverelt}}
\newcommand{\indcovereltclosureinterior}{\overline{\indcoverelt}^{\circ}}
\newcommand{\indcovereltprime}{\indcoverelt'}
\newcommand{\indcovereltprimeclosure}{\overline{\indcovereltprime}}
\newcommand{\DiameterBound}{\delta}

\newcommand{\distance}{d}
\newcommand{\Ball}{B}
\newcommand{\BallZ}[2]{\Ball_{#2}(#1)}

\newcommand{\FS}{FS}
\newcommand{\distanceFS}{\distance_{\FS}}
\newcommand{\flow}{\Phi}
\newcommand{\GPer}{\text{per}^{\Group}_{\flow}}
\newcommand{\SplittingConst}{\gamma}
\newcommand{\FSshort}{\FS_{\leq \SplittingConst}}
\newcommand{\FSlong}{\FS_{> \SplittingConst}}
\newcommand{\FSnonstationary}{\FS \backslash \FS^{\RR}}

\newcommand{\Rips}{X}
\newcommand{\RipsCpct}{\overline{\Rips}}

\newcommand{\SimpCx}{K}
\newcommand{\SimpCover}{\mathcal{C}}
\newcommand{\simpcoverelt}{C}
\newcommand{\SimpCxBarycentric}{b\SimpCx}
\newcommand{\barycentre}[1]{\hat{#1}}
\newcommand{\Star}[1]{\text{star}_{#1}}
\newcommand{\openstar}{\Star{\SimpCx}}
\newcommand{\barycentricstar}[1]{\Star{\SimpCxBarycentric}(\barycentre{#1})}
\newcommand{\simplexint}{\text{int}}

\newcommand{\Nerve}{\mathcal{N}}
\newcounter{commentcounter}

\title{Equivariant Refinements}
\author{Adam Mole\thanks{Supported by SFB 878 - Groups, Geometry and Actions.} \and Henrik R{\"u}ping\thanks{Supported by the Leibniz prize of Prof. Dr. Wolfgang L{\"u}ck granted by the deutsche Forschungsgemeinschaft.}}
\begin{document}

\maketitle

\begin{abstract}
 We show that if an open cover of a finite dimensional space is equivariant with respect to some finite group action on the space
  then there is an equivariant refinement of bounded dimension. 
This will generalize some constructions of certain covers. Those generalizations play a key role in the proof of the Farrell-Jones conjecture for the general linear group over a finite field.
\end{abstract}

\section{Introduction}

If a metric space~$\Space$ has covering dimension~$\leq \dimn$ then any open cover of~$\Space$ has a refinement of dimension~$\leq \dimn$,
  i.e. a refinement such that every point of~$\Space$ is contained in at most~$\dimn + 1$ sets from the refinement.

If a group~$\Group$ acts on~$\Space$ then there is an induced action of~$\Group$ on the set of subsets of~$\Space$, and we can call a cover~$\CoverArb$ of~$\Space$ \emph{equivariant} if is invariant under this induced action.
Then define a~$\Group$-cover of~$\Space$ to be an equivariant cover of~$\Space$ such that every translate of every set~$\coverelt$ in the cover is either disjoint from or coincides with~$\coverelt$.

It then seems a natural question to ask if, given a $\Group$-space~$\GSpace$ of covering dimension~$\leq \dimn$, does every $\Group$-cover of~$\GSpace$ have a $\Group$-refinement of dimension~$\leq \dimn$? 

In this article we show that for a finite group~$\GroupFin$ acting by isometries on a metric space~$\FMetricSpace$ there is always an $\GroupFin$-refinement, i.e. a refinement that is itself an $\GroupFin$-cover.
This will be done in Section~\ref{Section:EquivariantRefinements}.

\subsection{Application To The Farrell-Jones conjecture}

A family $\Family$ of subgroups of a group $\Group$ is a set of subgroups closed under conjugation and taking subgroups. 
In \cite{BartelsLueckReich2008KFJC}, Bartels-L\"uck-Reich proved an axiomatic formulation of the $K$-theoretic Farrell-Jones conjecture for a finitely generated group~$\Group$, with an arbitrary family~$\Family$ of subgroups of~$\Group$.
One of the conditions in their formulation is the existence of a metric $\Group$\nobreakdash-\hspace{0pt}space~$\GMetricSpace$ with a compactification~$\GMetricSpaceCpct$ of~$\GMetricSpace$ such that~$\Group \times \GMetricSpaceCpct$ has \emph{wide open $\Family$-covers} (see~\cite{BartelsLueckReich2008KFJC}*{Assumption 1.4}).

For the special case that the group~$\Group$ is hyperbolic, such covers for the compactification~$\RipsCpct$ of the Rips complex were constructed for the family~$\Vcyc$ of virtually cyclic subgroups in~\cite{BartelsLueckReich2008ECHG}.
There they constructed a flow space~$\FS$ together with a map~$\Group \times \GMetricSpaceCpct \rightarrow \FS$.
They then used the flow on~$\FS$ to carefully construct a cover of~$\FS$, which they pull-back to~$\Group \times \GMetricSpaceCpct$ to get a cover with the desired properties.

However, the construction uses only certain properties that can be derived from the hyperbolicity of the group and the choice of~$\GMetricSpace$ and~$\GMetricSpaceCpct$.
The full list of conditions was stated in~\cite{BartelsLueck2010CAT0Flow}*{Convention 5.1}.
One of these properties is the existence of a bound on the order of finite subgroups of~$\Group$, and it is this condition we will show is unnecessary by adapting the construction of the covers on the flow space.

If a group acts properly, isometrically and cocompactly on a CAT(0) space, it has such a bound. In \cite{BartelsLueckReichRueping2012GLnZ} the cocompactness condition was weakened to show the Farrell-Jones conjecture for $\GLnZ$. A group satisfying this weaker condition need not have a bound on the order of finite subgroups, however $\GLnZ$ still does.  For a finite field the groups $\GLnFt$ still satisfy those weaker conditions, as shown in \cite{HenrikThesis}. However these groups do not have a bound on the order of finite subgroups. Removing this bound is a crucial step in the proof of the Farrell-Jones conjecture for $S$-arithmetic groups over function fields. 

%This construction in~\cite{BartelsLueckReich2008ECHG} will be summarised in section~\ref{Section:ConstructingLongThinCovers} with a special emphasis on where the bound on the order of finite subgroups is used.
%Then we will be able to generalise this construction in section~?????

\section{Equivariant Refinements}\label{Section:EquivariantRefinements}

First we recall the definition of the covering dimension of a topological space.
A space has dimension $\le n$ if and only if every finite open cover has a finite $n$-dimensional refinement (see \cite{Pears1975DimensionTheory}*{Definition~3.1.1}). However we have 

 \begin{theorem}[\cite{Pears1975DimensionTheory}*{Theorem~3.4.3}]\label{Thm:OstrandsThm} A normal space  has dimension $\le n$, if and only if every locally finite open cover has an open refinement of dimension $\le n$.
\end{theorem}

Recall that a space is called \emph{paracompact} if every open cover has a locally finite open refinement. For example, metric spaces are paracompact by \cite{Pears1975DimensionTheory}*{Corollary~2.1.8} and normal. From Theorem~\ref{Thm:OstrandsThm} we immediately get the following corollary.

\begin{corollary}\label{Cor:CoveringDimForParacompactNormal} A paracompact, normal space has dimension $\le n$, if and only if any open cover has a refinement of dimension $\le n$.
\end{corollary}

So in a metric space of dimension~$\leq n$ we can find an $n$-dimensional refinement of an arbitrary cover.

Now we formally give the definition of a $\Group$-cover;
\begin{definition}\label{Defn:EquivariantCover}
 Let~$\Group$ be a group and let~$\GSpace$ be a topological space with a $\Group$-action.
 A cover~$\CoverArb$ of~$\GSpace$ is \emph{$\Group$-equivariant} if
 \[
  \forall \groupelt \in \Group, \, \forall \coverelt \in \CoverArb, \quad \groupelt \cdotp \coverelt \in \CoverArb.
 \]
 A \emph{$\Group$-cover} is a $\Group$-equivariant cover~$\CoverArb$ that also satisfies
 \[
  \forall \groupelt \in \Group, \, \forall \coverelt \in \CoverArb, \quad \groupelt \cdotp \coverelt \cap \coverelt \neq \emptyset \Rightarrow \groupelt \cdotp \coverelt = \coverelt.
 \]
 A \emph{$\Group$-refinement} of a cover is a refinement that is itself a $\Group$-cover.
\end{definition}

The goal of this section is to prove that for any finite group~$\GroupFin$ acting by isometries on an $\dimn$-dimensional metric space~$\FMetricSpace$, every open $\GroupFin$-cover of~$\FMetricSpace$ has an $\GroupFin$-refinement whose dimension is~$\leq \dimn$.

The idea is to project the cover down to the quotient space~$\GroupFin \backslash \FMetricSpace$ and take a refinement there, which can then be pulled back to give an equivariant cover of~$\FMetricSpace$ before taking careful intersections with the original cover to obtain an equivariant refinement.

For this we need to know the dimension of the quotient space.
We can use a general result about continuous open mappings in dimension theory.

\begin{proposition}[\cite{Pears1975DimensionTheory}*{Proposition 9.2.16}]\label{Propn:OpenCtsMappingDimensions}
 Let~$X,Y$ be weakly paracompact $T_4$-spaces.
 Let~$f \colon X \rightarrow Y$ be a continuous, open surjection.
 If for every point~$y \in Y$ the pre-image~$f^{-1}(y)$ is finite then
 \[
  \dim(X) = \dim(Y).
 \]
\end{proposition}

We have stated the result in the full generality given in the book,
 but every metric space is a weakly paracompact $T_4$-space, so if you are unfamiliar with the definition of weakly paracompact or a $T_4$-space then you can take~$X,Y$ to be metric spaces.

To apply Proposition~\ref{Propn:OpenCtsMappingDimensions} to the quotient map~$\FMetricSpace \rightarrow \FMetricSpaceQuotient$ we need to know that the quotient space is a weakly paracompact $T_4$-space.
It suffices to show it is a metric space.
For any finite group~$\GroupFin$ acting by isometries on a metric space~$\FMetricSpace$ the quotient~$\FMetricSpaceQuotient$ inherits a metric via
\[
 \distance_{\FMetricSpaceQuotient} ([\fmetricspaceelt],[\fmetricspaceeltprime]) :=
   \min_{\groupfinelt,\groupfineltprime \in \GroupFin} ( \groupfinelt \fmetricspaceelt , \groupfineltprime \fmetricspaceeltprime ).
\]
Hence Proposition~\ref{Propn:OpenCtsMappingDimensions} immediately gives;

\begin{corollary}\label{Cor:DimensionOfQuotient}
 For any finite group~$\GroupFin$ and any metric space~$\FMetricSpace$, if $\GroupFin$ acts on~$\FMetricSpace$ via isometries then
 \[
  \dim(\FMetricSpace) = \dim(\FMetricSpaceQuotient).
 \]
\end{corollary}

We can now prove the existence of open $\GroupFin$-refinements.
\begin{proposition}\label{Propn:EquivariantRefinements}
 Let~$\FMetricSpace$ be a metric space with~$\dim(\FMetricSpace) \leq \dimn$.
 Let~$\GroupFin$ be a finite group that acts by isometries on~$\FMetricSpace$.
 Any open $\GroupFin$-cover~$\CoverArb$ of~$\FMetricSpace$ has an open $\GroupFin$\nobreakdash-\hspace{0pt}refinement~$\Refinement$ with~$\dim(\Refinement) \leq \dimn$.
\end{proposition}
\begin{proof}
 Let~$\pi \colon \FMetricSpace \rightarrow \FMetricSpaceQuotient$ be the projection map, which is an open, continuous surjection.
 Then the collection~$\pi(\CoverArb) := \{ \pi(\coverelt) \, | \, \coverelt \in \CoverArb \}$ is an open cover of~$\FMetricSpaceQuotient$.
 By corollaries~\ref{Cor:CoveringDimForParacompactNormal} and~\ref{Cor:DimensionOfQuotient} we know there is a refinement~$\RefQuotient$ of~$\pi(\CoverArb)$ with~$\dim(\RefQuotient) \leq \dimn$.

 The pull-back~$\pi^{-1}(\RefQuotient)$ is an equivariant cover of~$\FMetricSpace$ but it is not necessarily a refinement of~$\CoverArb$.
 Taking intersections with elements of~$\CoverArb$ would give a refinement of~$\CoverArb$ but we need to be careful about the dimension and to not lose the equivariance.

 So for every~$\refquotelt \in \RefQuotient$ fix an element~$\refquoteltsuperset \in \CoverArb$ such that~$\refquotelt \subseteq \pi(\refquoteltsuperset)$.
 Now define
 \[
  \Refinement = \{ \refeltexpression \, | \, \refquotelt \in \RefQuotient, \, \groupfineltother \in \GroupFin \}.
 \]
 Every element of~$\Refinement$ is open since it is the intersection of two open sets.
 The collection~$\Refinement$ covers~$\FMetricSpace$ since for any point~$\fmetricspaceelt \in \FMetricSpace$ there is some~$\refquotelt \in \RefQuotient$ with~$\pi(\fmetricspaceelt) \in \refquotelt \subseteq \pi(\refquoteltsuperset)$ and then~$\fmetricspaceelt \in \pi^{-1}(\refquotelt) \subseteq \pi^{-1}(\pi(\refquoteltsuperset)) = \GroupFin \cdotp \refquoteltsuperset$.

 To show it is equivariant, consider~$\refelt = \refeltexpression$.
 For any~$\groupfinelt \in \GroupFin$ we have
 \begin{align*}
  \groupfinelt \cdotp \refelt  &=  \groupfinelt \cdotp ( \refeltexpression )
      \\      &=  \groupfinelt \cdotp \pi^{-1}(\refquotelt) \cap (\groupfinelt \groupfineltother) \cdotp \refquoteltsuperset
      \\      &=  \pi^{-1}(\refquotelt) \cap (\groupfinelt\groupfineltother) \cdotp \refquoteltsuperset \in \Refinement.
 \end{align*}
 There is one condition that still needs to be proven to know it is an open $\GroupFin$-refinement, namely that translates are disjoint (or unmoved).

 Again consider~$\refelt = \refeltexpression$.
 The set~$\pi^{1} (\refquotelt)$ is invariant under all of~$\GroupFin$ so we only need to look at $\groupfineltother \cdotp \refquoteltsuperset$.
 If~$\groupfinelt \cdotp \refelt \cap \refelt \neq \emptyset$ for some~$\groupfinelt \in \GroupFin$ then~$\groupfinelt \cdotp (\groupfineltother \cdotp \refquoteltsuperset) \cap \groupfineltother \cdotp \refquoteltsuperset \neq \emptyset$ and thus~$\groupfinelt \cdotp (\groupfineltother \cdotp \refquoteltsuperset) = \groupfineltother \cdotp \refquoteltsuperset$, since $\groupfineltother \cdotp \refquoteltsuperset \in \CoverArb$, which is an $\GroupFin$-cover.

 It remains to prove that~$\dim(\Refinement) \leq \dimn$.
 We do this by proving that $\dim(\Refinement) \leq \dim(\RefQuotient)$.
 Fix any~$\fmetricspaceelt \in \FMetricSpace$.
 The idea is to show that the projection~$\pi$ induces an injection
 \[
  \pi_{\fmetricspaceelt} \colon \{ \refelt \in \Refinement \, | \, \fmetricspaceelt \in \refelt \} \rightarrow \{ \refquotelt \in \RefQuotient \, | \, \pi(\fmetricspaceelt) \in \refquotelt \}, \refelt \mapsto \pi(\refelt).
 \]
 First we need to show that such a $\pi_{\fmetricspaceelt}$ is well-defined, i.e. that any element $\refelt = \refeltexpression$ is mapped to an element of~$\RefQuotient$.
 The projection is surjective and by definition~$\refquotelt \subseteq \pi(\refquoteltsuperset) = \pi(\groupfineltother \cdotp \refquoteltsuperset)$ so~$\pi(\refelt) = \refquotelt \in \RefQuotient$.
 Hence the map~$\pi_{\fmetricspaceelt} \colon \refelt \mapsto \pi(\refelt)$ is well-defined.

 It remains to prove that it is injective.
 Suppose~$\pi_{\fmetricspaceelt}(\refelt) = \pi_{\fmetricspaceelt}(\refelt')$.
 Then write~$\refelt = \refeltexpression$ and~$\refeltprime = \refeltprimeexpression$.
 Since~$\pi(\refelt) = \refquotelt$ and similarly for~$\refeltprime$, we know~$\refquotelt = \refquoteltprime$.
 By assumption~$\fmetricspaceelt \in \groupfineltother \cdotp \refquoteltsuperset \cap \groupfineltotherprime \cdotp \refquoteltsuperset$ but $\CoverArb$ is an $\GroupFin$-cover so we conclude~$\groupfineltother \cdotp \refquoteltsuperset = \groupfineltotherprime \cdotp \refquoteltsuperset$ and thus~$\refelt = \refeltprime$.
 Therefore the map is injective.

 This shows that the dimension of~$\Refinement$ is bounded by the dimension of~$\RefQuotient$.
 In particular~$\dim(\Refinement) \leq \dimn$.
\end{proof}

The refinement constructed in this proof is not canonical because it will depend on the choices of the~$\refquoteltsuperset$ for the~$\refquotelt \in \RefQuotient$.

\section{Generalising the General Position Arguments}

The motivation for looking at these $\GroupFin$-covers was to remove the condition that there is a bound on the order of finite subgroups in~\cite{BartelsLueck2010CAT0Flow}*{Convention~5.1}.
The convention we will use here given below.
\begin{convention}\label{Conventions}
 We make the following assumptions:
 \begin{itemize}
  \item $G$ is a discrete group,
  \item $\Family$ is a family of subgroups of~$G$,
  \item $(\FS,\distanceFS)$ is a metric space with an action of~$G$ via isometries,
  \item $\flow \colon \FS \times \RR \rightarrow \FS$ is a $G$-equivariant flow,
  \item $\FSnonstationary$ is locally connected, locally compact and finite dimensional,
  \item the action of $G$ on~$\FSnonstationary$ is proper
  \item the flow is uniformly continuous, i.e. for all~$\alpha >0$ and all~$\epsilon >0$ there exists some~$\delta >0$ such that for any~$z,z' \in \FS$ and any~$\tau \in [-\alpha,\alpha]$
  \[
   \distanceFS(z,z') \leq \delta
   \Rightarrow
   \distanceFS \big( \flow_{\tau}(z) , \flow_{\tau}(z') \big) \leq \epsilon.
  \]
 \end{itemize}

 Note that~$\FS^{\RR}$ is $G$-invariant since the flow is $G$-equivariant, so there is a well-defined action of~$G$ on the subspace~$\FSnonstationary$.
\end{convention}
 
Our assumptions here differ from \cite{BartelsLueck2010CAT0Flow}*{Convention~5.1} in a couple of ways. Firstly, we do not assume there is a bound on the order of finite subgroups. Secondly, we do not ask for the $G$-action to be proper on all of~$\FS$, only away from the stationary points. Finally, we do not require all of~$\FS$ to be locally compact, only~$\FSnonstationary$.
 
This is because we are working towards a more general version of \cite{BartelsLueck2010CAT0Flow}*{Theorem~5.6}.
Before we can state the theorem we need to explain some notation.
The \emph{$G$-period} of an element $z \in \FS$ is
\[
 \GPer(z) := \inf \lbrace \tau \in [0,\infty] \colon \exists g \in G \text{ with } \flow_{\tau} (z) = g z \rbrace
\]
where the infimum of the empty set is defined to be~$\infty$.
Then for any~$\gamma \geq 0$ we can split the flow space up into two parts;
\[
 \FSshort := \lbrace z \in \FS \, \vert \, \GPer(z) \leq \gamma \rbrace,
\]
\[
 \FSlong := \lbrace z \in \FS \, \vert \, \GPer(z) > \gamma \rbrace.
\]
We also need the concept of an $\Family$-subset and an $\Family$-cover.
An \emph{$\Family$-subset of~$\FS$} is a subset~$U \subseteq \FS$ such that
\begin{itemize}
 \item $\forall g \in G, \, \quad g \cdotp U \cap U \neq \emptyset \Rightarrow g \cdotp U = U$,
 \item $G_U \in \Family$.
\end{itemize}
Then an \emph{$\Family$-collection of subsets of~$\FS$} is a $G$-equivariant collection~$\mathcal{U}$ of $\Family$\nobreakdash-\hspace{0pt}subsets of~$\FS$.

The following theorem is a stronger version of \cite{BartelsLueck2010CAT0Flow}*{Theorem~5.6} and tells us about the existence of certain `nice' covers of cocompact subspaces of~$\FSlong$.

\begin{theorem}\label{Thm5.6}
 Under the assumptions of Convention~\ref{Conventions}, there exists a number $M \in \NN$ such that for any~$\alpha>0$ there is a~$\gamma>0$  such that for any compact~${K \subseteq \FSlong}$ there is a collection~$\IndCover$ of open $\Vcyc$-subsets of~$\FS$ satisfying
 \begin{enumerate}
  \item $\IndCover$ is $G$-invariant, i.e. for all~$g \in G$ and all~$\indcoverelt \in \IndCover$, the set~$g\indcoverelt$ is also in~$\IndCover$,
  \item $\dim \IndCover \leq N$,
  \item $G \backslash \IndCover$ is finite,
  \item for any~$z \in G K$ there is some set~$\indcoverelt \in \IndCover$ with~$\flow_{[-\alpha,\alpha]}(z) \subseteq \indcoverelt$.
 \end{enumerate}
\end{theorem}

The proof of this theorem is as for the proof of \cite{BartelsLueck2010CAT0Flow}*{Theorem~5.6}.
We are working away from the stationary points so it does not matter that we weakened a couple of assumptions from \cite{BartelsLueck2010CAT0Flow}*{Convention~5.1}, where we allow things to be a bit more wild on~$\FS^{\RR}$.
However, we do need to prove that the bound on the order of finite subgroups is not necessary.

\cite{BartelsLueck2010CAT0Flow}*{Theorem~5.6} is a more precise formulation of \cite{BartelsLueckReich2008ECHG}*{Proposition~4.1} so it suffices to prove that the bound on the order of finite subgroups was unnecessary for \cite{BartelsLueckReich2008ECHG}*{Proposition~4.1}.

The proof of \cite{BartelsLueckReich2008ECHG}*{Proposition~4.1} is long and technical, but the only time the bound on the order of finite subgroups is used is in the general position arguments in \cite{BartelsLueckReich2008ECHG}*{Section~3}.
In particular, we need to strengthen \cite{BartelsLueckReich2008ECHG}*{Propositions~3.2 and~3.3} so that they do not depend on the order of the finite subgroup.
The proof of our stronger versions (Propositions~\ref{Propn:Stronger32} and~\ref{Propn:Stronger33} below) will follow their proof but it will be made clear where we need to be more careful.

We first strengthen~\cite{BartelsLueckReich2008ECHG}*{Proposition~3.3} so that it does need depend on the order of the finite subgroup.
A quick couple of remarks about this proposition.
Firstly, we give a very slightly weaker property~$(\ref{Item:IndCoverSelfIntersections})$ than in~\cite{BartelsLueckReich2008ECHG}*{Proposition 3.3}, namely our bound is~$2^{j+1}-1$ instead of~$2^{j+1}-2$.
This does not cause any problems since all that is used is that the bound is~$< 2^{j+1}$.

Secondly, we do not require there to be a locally connected subspace~$Y$ that contains all the~$\overline{\collectionelt}$.
The existence of the subspace was only an extra condition and did not appear in the conclusions of \cite{BartelsLueckReich2008ECHG}*{Proposition~3.3} and we are able to remove it from our assumptions.
However, the assumption in Convention~\ref{Conventions} that $\FSnonstationary$ is locally connected cannot be removed because it is also used elsewhere in the proof of \cite{BartelsLueckReich2008ECHG}*{Proposition~4.1}.
Namely, it is used to make the central slices of the boxes $S_{C_{\lambda}}$ connected and this fact is used in the proof of \cite{BartelsLueckReich2008ECHG}*{Lemma~4.20}.
It may be possible to avoid this usage so we give a more general version of \cite{BartelsLueckReich2008ECHG}*{Proposition~3.3} than is strictly necessary for us to deduce Theorem~\ref{Thm5.6}.

\begin{proposition}\label{Propn:Stronger33}
 Let~$\FMetricSpace$ be a compact metric space with covering dimension~$\dimn$.
 Let~$\GroupFin$ be a finite group that acts on~$\FMetricSpace$ by isometries.
% Let~$\FMetricSpaceSubset \subseteq \FMetricSpace$ be an open, \textcolor{red}{\sout{locally connected,}} $\GroupFin$-invariant subset of~$\FMetricSpace$. \textcolor{blue}{I don't know if $Y$ is necessary any more or if it can be completely removed from the proposition.}
 Let~$\BoundaryBound \in \NN$ be a fixed integer.
 Let~$\Collection$ be a finite $\GroupFin$-equivariant collection of open subsets of~$\FMetricSpace$ such that
% \begin{itemize}
%  \item $\forall \collectionelt \in \Collection, \quad \collectioneltclosure \subseteq \FMetricSpaceSubset$;
%  \item $\forall \Collection_0 \subseteq \Collection, \quad | \Collection_0 | > \BoundaryBound \Rightarrow \bigcap_{\collectionelt \in \Collection_0} \partial\collectionelt = \emptyset$.
% \end{itemize}
 \[
  \forall \Collection_0 \subseteq \Collection, \quad | \Collection_0 | > \BoundaryBound \Rightarrow \bigcap_{\collectionelt \in \Collection_0} \partial\collectionelt = \emptyset.
 \]
  Then for any~$\DiameterBound >0$ there are finite collections~$\IndCover^0, \ldots , \IndCover^{\dimn}$ of open subsets of~$\FMetricSpace$ satisfying for all~$j=0,\ldots,n$
  % The part I cut out to make the itemize tidier
  %\forall j \in \{ 0 , \ldots , \dimn \}, \, 
 \begin{enumerate}[(i)]
  \item\label{Item:IndCoverPartition}
    $\IndCover := \IndCover^0 \cup \ldots \cup \IndCover^{\dimn}$ is an open cover of~$\FMetricSpace$;
  \item\label{Item:IndCoverBounded}
    $\forall \indcoverelt \in \IndCover, \quad \diam(\indcoverelt) < \DiameterBound$;
  \item\label{Item:IndCoverIntersectingU}
    $\forall \indcoverelt \in \IndCover, \quad | \{ \collectionelt \in \Collection \, | \, \collectionelt \cap \indcoverelt \neq \emptyset \} | \leq \BoundaryBound$;
  \item\label{Item:IndCoverSelfIntersections}
    $\forall \indcoverelt_0 \in \IndCover^j, \quad | \{ \indcoverelt \in \IndCover^0 \cup \ldots \IndCover^j \, | \, \indcoverelt_0 \cap \indcoverelt \neq \emptyset \} | \leq 2^{j+1} - 1$;
  \item\label{Item:IndCoverDisjointClosures}
    $\forall \indcoverelt,\indcovereltprime \in \IndCover^j, \quad
             \indcoverelt \neq \indcovereltprime \, \Rightarrow \, \indcovereltclosure \cap \indcovereltprimeclosure = \emptyset$;
  \item\label{Item:IndCoverInvariance}
    $\forall \indcoverelt \in \IndCover^j, \, \forall \groupfinelt \in \GroupFin, \quad \groupfinelt \indcoverelt \in \IndCover^j$;
  \item\label{Item:IndCoverIntClosure}
    $\indcoverelt \in \IndCover, \quad \indcovereltclosureinterior = \indcoverelt$.
 \end{enumerate}
\end{proposition}

The idea behind the proof is to show that a simplicial complex has a canonical cover satisfying properties~$(\ref{Item:IndCoverPartition})$ and~$(\ref{Item:IndCoverSelfIntersections})$, as well as a weaker version of~$(\ref{Item:IndCoverDisjointClosures})$.
Then we construct an $\GroupFin$-cover~$\Refinement$ of~$\FMetricSpace$ satisfying properties~$(\ref{Item:IndCoverBounded})$,~$(\ref{Item:IndCoverIntersectingU})$, and~$(\ref{Item:IndCoverInvariance})$, and pull-back the canonical cover of the nerve~$\Nerve(\Refinement)$  of~$\Refinement$ to~$\FMetricSpace$.
Finally we shrink this cover so that properties~$(\ref{Item:IndCoverDisjointClosures})$ and~$(\ref{Item:IndCoverIntClosure})$ are fulfilled.

Before that, some notation.
If~$\SimpCx$ is a simplicial complex then let~$|\SimpCx|$ be the geometric realisation of~$\SimpCx$.
For a simplex~$\sigma$ of~$\SimpCx$, let~$\simplexint(\sigma)$ be the interior of~$\sigma$ in~$|\SimpCx|$.
Let~$\openstar(\sigma)$ be the open star of~$\sigma$ in~$|\SimpCx|$, i.e. the union of the interiors of all the simplices of~$\SimpCx$ that contain~$\sigma$.
Let~$\SimpCxBarycentric$ be the barycentric subdivision of~$\SimpCx$ and let~$\barycentre{\sigma}$ be the barycentre of~$\sigma$ (so~$\barycentre{\sigma}$ is a vertex of~$\SimpCxBarycentric$.)
Note that~$|\SimpCx| = |\SimpCxBarycentric|$.

Now we begin by constructing a canonical cover of an arbitrary simplicial complex of dimension~$\leq \dimn$.

\begin{lemma}\label{Lemma:CanonCoverOfSimpCx}
 Let~$\SimpCx$ be a simplicial complex of dimension~$\dimn$.
 There are finite collections~$\SimpCover^0, \ldots, \SimpCover^{\dimn}$ of subsets of~$|\SimpCx|$ such that for all~$j = 0,\ldots,n$
 \begin{itemize}
  \item $\SimpCover := \SimpCover^0 \cup \ldots \cup \SimpCover^{\dimn}$ is an open cover of~$|\SimpCx|$;
  \item $\forall \simpcoverelt_0 \in \SimpCover^j, \quad | \{ \simpcoverelt \in \SimpCover^0 \cup \ldots \SimpCover^j \, | \, \simpcoverelt_0 \cap \simpcoverelt \neq \emptyset \} | \leq 2^{j+1} - 1$;
  \item $\forall \simpcoverelt,\simpcoverelt' \in \SimpCover^j, \quad \simpcoverelt \neq \simpcoverelt' \Rightarrow \simpcoverelt \cap \simpcoverelt' = \emptyset$.
 \end{itemize}
\end{lemma}
\begin{proof}
 Any point~$x \in |\SimpCx|$ is contained in the interior of a (unique) simplex~$\sigma_x$ and so~$x \in \simplexint(\sigma_x) \subseteq \barycentricstar{\sigma_x}$. 
 Hence for~$j = 0, \ldots, \dimn$ set
 \[
  \SimpCover^j := \{ \barycentricstar{\sigma} \, \colon \, \dim(\sigma) = j \}.
 \]
 and then~$\SimpCover := \SimpCover^0 \cup \ldots \cup \SimpCover^{\dimn}$ is an open cover of~$|\SimpCx|$.
 We need to bound the number of intersections.

 Fix~$j \in \{ 0, \ldots, \dimn\}$ and then an element of~$\SimpCover^j$ is of the form~$\barycentricstar{\sigma}$ for some simplex~$\sigma$ of~$\SimpCx$ with~$\dim(\sigma) = j$.
 If~$\tau$ is another simplex of~$\SimpCx$ such that~$\barycentricstar{\sigma} \cap \barycentricstar{\tau} \neq \emptyset$ then there is some simplex~$\eta$ of~$\SimpCxBarycentric$ which contains~$\hat{\sigma}$ and~$\hat{\tau}$ as vertices.
So by the definition of the barycentric subdivision, we have~$\sigma \subseteq \tau$ or~$\tau \subseteq \sigma$.

 This means that any element of~$\SimpCover^0 \cup \ldots \cup \SimpCover^j$ that intersects~$\barycentricstar{\sigma}$ non-trivially corresponds to a face of~$\sigma$.
 Therefore the number of intersections is bounded by the number of faces of~$\sigma$ (including~$\sigma$ itself), which is~$2^{\dim(\sigma)+1} - 1$.
 Remembering that~$\dim(\sigma) = j$ we obtain the desired bound.

 In particular, this argument also shows that no two elements of~$\SimpCover^j$ can intersect non-trivially.
\end{proof}

Now that we know we can construct an appropriate cover of a simplicial complex we need to decide which simplicial complex to use.
Our choice will be based on the following crucial observation.
\begin{lemma}\label{Lemma:PullBackIsRefinement}
 Let~$\FMetricSpace$ be a metric space.
 Let~$\Refinement$ be an open cover of~$\FMetricSpace$.
 There is a map~$f_{\Refinement} \colon \FMetricSpace \rightarrow |\Nerve(\Refinement)|$ from~$\FMetricSpace$ to the (realisation of the) nerve~$\Nerve(\Refinement)$ of~$\Refinement$ such that the pull-back of the canonical cover of~$|\Nerve(\Refinement)|$ (given by lemma~\ref{Lemma:CanonCoverOfSimpCx}) is a refinement of~$\Refinement$.
\end{lemma}
\begin{proof}
 The continuous map from~$\FMetricSpace$ to~$\Nerve(\Refinement)$ is given by
 \[
  f_{\Refinement} \colon \FMetricSpace \rightarrow |\Nerve(\Refinement)| ,
  \quad z \mapsto \sum_{\refelt \in \Refinement} \frac{\distance(z,\FMetricSpace \backslash \refelt)}{\sum_{\refelt' \in \Refinement} \distance(z, \FMetricSpace \backslash \refelt')} [\refelt].
 \]
 Informally, this maps a point to a (weighted) sum of the elements of~$\Refinement$ that contain the point.
 (The weighting is to ensure the map is continuous.)
 In particular, the support of~$f_{\Refinement}(z)$ contains~$[\refelt]$ if and only if~$z \in \refelt$.
 Hence for any~$\refelt \in \Refinement$, if~$\openstar([\refelt])$ denotes the open star of the vertex~$[\refelt]$ in~$\Nerve(\Refinement)$ then~$f_{\Refinement}^{-1} (\openstar([\refelt])) = W$.
 
 Let~$\SimpCover$ be the canonical cover of~$|\Nerve(\Refinement)|$ given by lemma~\ref{Lemma:CanonCoverOfSimpCx}.
 Then the pull-back~$\hat{\IndCover} = f_{\Refinement}^{-1} (\SimpCover)$ is a cover of~$\FMetricSpace$.
 Any element of~$\SimpCover$ is of the form~$\barycentricstar{\sigma}$ with~$\sigma$ a simplex of~$\Nerve(\Refinement)$.
 If~$[\refelt]$ is a vertex of~$\sigma$ then $\barycentricstar{\sigma} \subseteq \openstar([\refelt])$, which follows from the definition of the barycentric subdivision.
 Hence
 \[
  f_{\Refinement}^{-1} \big( \barycentricstar{\sigma} \big) \subseteq
  f_{\Refinement}^{-1} \big( \openstar([\refelt]) \big)      =  \refelt
 \]
 and so we have shown that~$\hat{\IndCover}$ is a refinement of~$\Refinement$.
\end{proof}

Therefore if we take our simplicial complex to be the nerve of some cover~$\Refinement$ then the pull-back of the nerve's canonical cover will inherit some properties of~$\Refinement$. We call any property of a cover that is always inherited by refinements a \emph{smallness property}.
By choosing this cover~$\Refinement$ carefully we can ensure that properties~$(\ref{Item:IndCoverBounded})$,~$(\ref{Item:IndCoverIntersectingU})$, and~$(\ref{Item:IndCoverInvariance})$ are satisfied.

It is in this next lemma that we need to use the work of Section~\ref{Section:EquivariantRefinements}, and here we deviate from the proof of \cite{BartelsLueckReich2008ECHG}*{Proposition~3.3}.

\begin{lemma}\label{Lemma:CoverOfZ}
 Let~$\GroupFin$,~$\FMetricSpace$,~$\FMetricSpaceSubset$,~$\BoundaryBound$, and~$\CoverArb$ be as in proposition~\ref{Propn:Stronger33}.
 For any~$\DiameterBound >0$ there is an open cover~$\Refinement$ of~$\FMetricSpace$ such that
 \begin{description}
  \item[$(0)$] $\dim(\Refinement) = \dimn$;
  \item[$(\ref{Item:IndCoverBounded})$] $\forall \refelt \in \Refinement, \quad \diam(\refelt) < \DiameterBound$;
  \item[$(\ref{Item:IndCoverIntersectingU})$] $\forall \refelt \in \Refinement, \quad | \{ \collectionelt \in \Collection \, | \, \refelt \nsubseteq \collectionelt , \collectionelt \cap \refelt \neq \emptyset \} | \leq \BoundaryBound$;
  \item[$(\ref{Item:IndCoverInvariance})$] $\forall \refelt \in \Refinement , \, \forall \groupfinelt \in \GroupFin, \quad \groupfinelt \refelt \in \Refinement$.
 \end{description}
\end{lemma}
\begin{proof}
 The idea of this proof is to construct an $\GroupFin$-cover that satisfies the smallness conditions~$(\ref{Item:IndCoverBounded})$ and~$(\ref{Item:IndCoverIntersectingU})$, and then pass to a refinement using proposition~\ref{Propn:EquivariantRefinements}.
We construct our cover by finding an appropriate neighbourhood of every point in~$\FMetricSpace$.
 
 Fix~$z \in \FMetricSpace$.
 First suppose that~$z$ is not contained in the closure of any element of~$\Collection$.
 Since~$\Collection$ is finite, the set~$\bigcup_{\collectionelt \in \Collection} \overline{\collectionelt}$ is a closed set.
Thus there is some~$\epsilon_z >0$ such that~$\BallZ{z}{\epsilon_z} \subseteq \FMetricSpace \backslash ( \bigcup_{\collectionelt \in \Collection} \overline{\collectionelt} )$.
 We may take~$\epsilon_z < \frac{1}{2} \DiameterBound$ and then~$W_z := \BallZ{z}{\epsilon_z}$ has diameter~$<\DiameterBound$ and doesn't meet any~$\overline{\collectionelt}$.
 The set~$\bigcup_{\collectionelt \in \Collection} \overline{\collectionelt}$ is~$\GroupFin$-invariant so we may also pick the~$\epsilon_z$ small enough such that for any~$\groupfinelt \in \GroupFin$ we have~$\epsilon_{\groupfinelt z} = \epsilon_z$ and then~$W_{\groupfinelt z} = \groupfinelt W_z$.
 For disjoint orbits, the group~$\GroupFin$ is finite so we can pick~$\epsilon_z$ small enough that if~$\BallZ{z}{\epsilon_z} \cap \groupfinelt \BallZ{z}{\epsilon_z} \neq \emptyset$ for some~$\groupfinelt \in \GroupFin$ then~$\groupfinelt z = z$.
 Thus~$W_z$ is a suitable open neighbourhood of~$z$.
 
 Now suppose that there is some element of~$\Collection$ whose closure contains~$z$.
% \textcolor{red}{\sout{Then~$z \in \FMetricSpaceSubset$, which is locally connected.
% Hence there is a connected open neighbourhood~$A_z$ of~$z$ in~$\FMetricSpace \backslash (\bigcup_{\collectionelt \in \Collection, z \notin \partial\collectionelt} \partial\collectionelt)$.
% Since~$A_z$ is connected, for any~$\collectionelt \in \Collection$, if~$A_z \cap \collectionelt \neq \emptyset$ then either~$A_z \subseteq \collectionelt$ or~$A_z \cap \partial\collectionelt \neq \emptyset$, but we have chosen~$A_z$ such that if~$A_z \cap \partial\collectionelt \neq \emptyset$ then~$z \in \partial\collectionelt$.}}
% \textcolor{red}{Set
% \[
%  A_z =
%   \left(
%    \bigcap_{\stackrel{\collectionelt \in \Collection}{z \in \collectionelt}} U
%   \right)
%   \cap
%   \left(
%    \bigcap_{\stackrel{\collectionelt \in \Collection}{z \notin \overline{\collectionelt}}} \overline{\collectionelt}^{c}
%   \right)
%   \cap \FMetricSpace
% \]
 Set
 \[
  A^1_z = \bigcap_{\stackrel{\collectionelt \in \Collection}{z \in \collectionelt}} U \, ,
  \qquad
  A^0_z = \bigcap_{\stackrel{\collectionelt \in \Collection}{z \notin \overline{\collectionelt}}} \overline{\collectionelt}^{c}
 \]
 where~$\overline{\collectionelt}^{c}$ denotes the complement of the set~$\overline{\collectionelt}$ in~$\FMetricSpace$.
 By convention we say~$A^1_z = \FMetricSpace$ if $z$ does not lie in any $\collectionelt \in \Collection$ and similarly $A^0_z = \FMetricSpace$ if $z$ is contained in all the $\overline{\collectionelt}$ for $\collectionelt \in \Collection$.
 Both~$A^1_z$ and~$A^0_z$ are open sets by the finiteness of~$\Collection$, and non-empty since $z \in A^1_z \cap A^0_z =: A_z$.
 For any~$\collectionelt \in \Collection$, if $z \in \collectionelt$ then $A^1_z \subseteq \collectionelt$ and if $z \notin \overline{\collectionelt}$ then $A^0_z \cap \collectionelt = \emptyset$.
 Therefore if $A_z \cap \collectionelt \neq \emptyset$ but $A_z \nsubseteq \collectionelt$ then $z \in \partial \collectionelt$.
 Thus
 \[
  | \{ \collectionelt \in \Collection \, | \, A_z \nsubseteq \collectionelt , \collectionelt \cap A_z \neq \emptyset \} |
  =
  | \{ \collectionelt \in \Collection \, | \, z \in \partial\collectionelt \} |
  \leq \BoundaryBound.
 \]
 Pick~$B_z$ to be an open neighbourhood of~$z$ in~$\FMetricSpace$ such that for every~$\groupfinelt \in \GroupFin$ if~$\groupfinelt B_z \cap B_z \neq \emptyset$ then~$\groupfinelt z = z$.
 Such a neighbourhood exists since~$\GroupFin$ is finite (and~$\FMetricSpace$ is Hausdorff).
 Thus the orbit of~$B_z$ is disjoint. 
 
 If we set~$C_z = A_z \cap B_z \cap \BallZ{z}{\frac{1}{2}\DiameterBound}$ then we also have the diameter bounded by~$\DiameterBound$.
 We still need these neighbourhoods to be $\GroupFin$-equivariant, meaning $C_{\groupfinelt z} = \groupfinelt C_z$ for all~$\groupfinelt \in \GroupFin$.
 We can achieve this by setting $W_z = \bigcap_{\groupfinelt \in \GroupFin} \groupfinelt^{-1} C_{\groupfinelt z}$.
 
 Then~$\Refinement_0 := \{ W_z \, | \, z \in \FMetricSpace \}$ is an open cover of~$\FMetricSpace$ satisfying
 \begin{itemize}
  \item $\forall \refelt \in \Refinement_0, \quad \diam(\refelt) < \DiameterBound$;
  \item $\forall \refelt \in \Refinement_0, \quad | \{ \collectionelt \in \Collection \, | \, \refelt \nsubseteq \collectionelt , \collectionelt \cap \refelt \neq \emptyset \} | \leq \BoundaryBound$;
  \item $\forall \refelt \in \Refinement_0 , \, \forall \groupfinelt \in \GroupFin, \quad \groupfinelt \refelt \in \Refinement$.
 \end{itemize}
 In fact,~$\Refinement_0$ is an open $\GroupFin$-cover with disjoint orbits and so we can apply proposition~\ref{Propn:EquivariantRefinements} to this cover to obtain an $\GroupFin$-refinement~$\Refinement$ which has dimension~$\leq \dimn$ and inherits the two smallness conditions from~$\Refinement_0$.
\end{proof}

Using these three lemmas we can prove proposition~\ref{Propn:Stronger33}.

\begin{proof}[Proof of Proposition~\ref{Propn:Stronger33}]
 Lemma~\ref{Lemma:CoverOfZ} gives us a cover~$\Refinement$ of~$\FMetricSpace$ that satisfies properties~$(\ref{Item:IndCoverBounded})$,~$(\ref{Item:IndCoverIntersectingU})$ and~$(\ref{Item:IndCoverInvariance})$.
 The nerve~$\Nerve(\Refinement)$ of~$\Refinement$ is then an $\dimn$\nobreakdash-\hspace{0pt}dimensional simplicial complex and the group action on~$\Refinement$ induces an action of~$\GroupFin$ on~$\Nerve(\Refinement)$.
 The (realisation of) the nerve has a canonical cover~$\SimpCover = \SimpCover^0 \cup \ldots \cup \SimpCover^{\dimn}$ given by lemma~\ref{Lemma:CanonCoverOfSimpCx}.
 The continuous map
 \[
  f_{\Refinement} \colon \FMetricSpace \rightarrow |\Nerve(\Refinement)|
  \quad z \mapsto \sum_{\refelt \in \Refinement} \frac{\distance(z,\FMetricSpace \backslash \refelt)}{\sum_{\refelt' \in \Refinement} \distance(z, \FMetricSpace \backslash \refelt')} [\refelt]
 \]
 from lemma~\ref{Lemma:PullBackIsRefinement} is $\GroupFin$-equivariant.
 For~$j = 0,\ldots,\dimn$ define~$\hat{\IndCover}^j = f_{\Refinement}^{-1} (\SimpCover^j)$.
 Then~$\hat{\IndCover} := \hat{\IndCover}^0 \cup \ldots \cup \hat{\IndCover}^{\dimn}$ is the pull-back of the canonical cover of~$|\Nerve(\Refinement)|$ and thus is a refinement of~$\Refinement$ by lemma~\ref{Lemma:PullBackIsRefinement}.
 From these definitions it immediately follows that~$\hat{\IndCover}$ satisfies property~$(\ref{Item:IndCoverPartition})$.

 Moreover, we claim that~$\hat{\IndCover}$ satisfies properties~$(\ref{Item:IndCoverBounded})$,~$(\ref{Item:IndCoverIntersectingU})$,~$(\ref{Item:IndCoverSelfIntersections})$, and~$(\ref{Item:IndCoverInvariance})$.
 
 Properties~$(\ref{Item:IndCoverBounded})$ and~$(\ref{Item:IndCoverIntersectingU})$ are smallness conditions which are satisfied by~$\Refinement$ so these properties are inherited by~$\hat{\IndCover}$.
 
 The cover~$\hat{\IndCover}$ satisfies property~$(\ref{Item:IndCoverSelfIntersections})$ since it is the pull-back of~$\SimpCover$ and~$\SimpCover$ satisfies property~$(\ref{Item:IndCoverSelfIntersections})$.
 
 The cover~$\hat{\IndCover}$ is $\GroupFin$-invariant (i.e. satisfies property~$(\ref{Item:IndCoverInvariance})$) since the cover~$\SimpCover$ is $\GroupFin$-invariant and the map~$f_{\Refinement}$ is $\GroupFin$-equivariant.

 However we do not necessarily have the remaining two properties~$(\ref{Item:IndCoverDisjointClosures})$ and~$(\ref{Item:IndCoverIntClosure})$.
 We do have a weaker version of property~$(\ref{Item:IndCoverDisjointClosures})$, namely we know that
 \[
  \forall j=0,\ldots,\dimn, \, \forall \indcoverelt,\indcoverelt' \in \hat{\IndCover}^j, \quad \indcoverelt \neq \indcoverelt' \Rightarrow \indcoverelt \cap \indcoverelt' = \emptyset
 \]
 because this is true of the cover~$\SimpCover$ of~$\Nerve(\Refinement)$.
 To get a cover of~$\FMetricSpace$ that also satisfies properties~$(\ref{Item:IndCoverDisjointClosures})$ and~$(\ref{Item:IndCoverIntClosure})$ we shrink the elements of the cover~$\hat{\IndCover}$, being careful to not lose any of the other properties, which we do as follows.
 
 For~$\indcoverelt \in \hat{\IndCover}$ and~$\epsilon >0$ set~$\indcoverelt^{-\epsilon} = \{ z \in \FMetricSpace \, | \, \BallZ{z}{\epsilon} \subseteq \indcoverelt \}$ and then set $\indcoverelt_{\epsilon} = (\overline{\indcoverelt^{-\epsilon}})^{\circ}$.
 The collection~$\{ \indcoverelt_{\sfrac{1}{m}} \, | \, \indcoverelt \in \hat{\IndCover}, \, m \in \NN \}$ is an open cover of~$\FMetricSpace$.
 By compactness there exists some~$m \in \NN$ such that~$\IndCover:= \{ \indcoverelt_{\sfrac{1}{m}} \, | \indcoverelt \in \hat{\IndCover} \}$ is an open cover of~$\FMetricSpace$.
 
 Now we claim that this satisfies all the properties we want.
 
 Set~$\IndCover^j = \{ \indcoverelt_{\sfrac{1}{m}} \, | \, \indcoverelt \in \hat{\IndCover}^j\}$ for property~$(\ref{Item:IndCoverPartition})$.
 The union of these~$\IndCover^j$ is~$\IndCover$ which is an open cover by the choice of~$m$.
 
 Properties~$(\ref{Item:IndCoverBounded})$ and~$(\ref{Item:IndCoverIntersectingU})$ are smallness conditions so~$\IndCover$ inherits them from~$\hat{\IndCover}$.
 Property~$(\ref{Item:IndCoverSelfIntersections})$ is not a smallness condition in general but we have not increased the number of elements in our cover and we cannot get more intersections by shrinking the sets so~$\IndCover$ also inherits property~$(\ref{Item:IndCoverSelfIntersections})$ from~$\hat{\IndCover}$.
 
 For property~$(\ref{Item:IndCoverDisjointClosures})$, observe that for any~$\indcoverelt \in \hat{\IndCover}$, the closure of~$\indcoverelt_{\sfrac{1}{m}}$ is contained in~$\indcoverelt$.
 Hence if the closures of~$\indcoverelt_{\sfrac{1}{m}},\indcoverelt'_{\sfrac{1}{m}} \in \IndCover^j$ intersect non-trivially then~$\indcoverelt \cap \indcoverelt' \neq \emptyset$ but then~$\indcoverelt = \indcoverelt'$ since~$\hat{\IndCover}$ satisfies a weaker version of property~$(\ref{Item:IndCoverDisjointClosures})$.
 
 The cover~$\hat{\IndCover}$ is $\GroupFin$-invariant and we have shrunk its elements in a uniform way, hence the cover~$\IndCover$ is also $\GroupFin$-invariant, i.e. it satisfies property~$(\ref{Item:IndCoverInvariance})$.
 
 Finally, every element of~$\IndCover$ is the interior of a closed set, thus condition~$(\ref{Item:IndCoverIntClosure})$ is fulfilled. 
\end{proof}

We still need to strengthen~\cite{BartelsLueckReich2008ECHG}*{Proposition~3.2}.
\begin{proposition}\label{Propn:Stronger32}
 Let~$\FMetricSpace$ be a compact metric space with covering dimension~$\dimn$.
 Let~$\GroupFin$ be a finite group that acts on~$\FMetricSpace$ by isometries.
 Let~$\Collection$ be a finite $\GroupFin$-equivariant collection of open $\GroupFin$-subsets of~$\FMetricSpace$.
 Given any map~$\alpha \colon \Collection \rightarrow \OpenSubsetsZ$, where~$\OpenSubsetsZ$ denotes the set of open subsets of~$\FMetricSpace$, if~$\overline{\alpha(\collectionelt)} \subseteq \collectionelt$ for all~$\collectionelt \in \Collection$ then there is an $\GroupFin$-equivariant map $\beta \colon \Collection \rightarrow \OpenSubsetsZ$ satisfying
 \begin{enumerate}
  \item $\forall \collectionelt \in \Collection, \, \alpha(\collectionelt) \subseteq \overline{\alpha(\collectionelt)} \subseteq \beta(\collectionelt) \subseteq \overline{\beta(\collectionelt)} \subseteq \collectionelt$,
  \item For any~$\Collection_0 \subseteq \Collection$, if~$\Collection_0$ contains more than~$\dimn+1$ elements then 
  \[
   \bigcap_{\collectionelt \in \Collection_0} \partial\beta(\collectionelt) = \emptyset.
  \]
 \end{enumerate}
\end{proposition}
\begin{proof}
 We won't give all the details here since almost all of the details are in \cite{BartelsLueckReich2008ECHG}.
 We will explain where they used the order of $\GroupFin$ and how to avoid this problem.
 The rest of the proof will only be sketched.
 
 First pick an $\GroupFin$-invariant metric~$d$ on~$\FMetricSpace$ and then by the compactness of~$\FMetricSpace$ and finiteness of~$\Collection$ there is some~$\delta >0$ such that for all~$\collectionelt \in \Collection$
 \[
  \overline{\alpha(\collectionelt)} \subseteq
  \lbrace
   \fmetricspaceelt \in \FMetricSpace
   \, \vert \,
   d ( \fmetricspaceelt , \FMetricSpace \backslash \collectionelt ) > \delta
  \rbrace
  =: \collectionelt^{-\delta}.
 \]
 
 As in the proof of Proposition~\ref{Propn:Stronger33}, before proving the proposition in full generality Bartels-L{\"u}ck-Reich consider the special case of a simplicial complex of dimension~$\dimn$ where each element of~$\Collection$ is the interior of a subcomplex.
 They use the barycentric subdivision to construct such a map~$\beta$.
 
 In the general case they start with an open cover~$\RefQuotient$ of~$\FMetricSpace$ with $\dim(\RefQuotient) \leq \dimn$ such that the diameter of any element of~$\RefQuotient$ is bounded by $\delta / 3$.
 To make this $\GroupFin$-equivariant they set~$\RefQuotient_{\GroupFin} = \GroupFin \RefQuotient = \lbrace \groupfinelt \refquotelt \, \vert \, \groupfinelt \in \GroupFin, \refquotelt \in \RefQuotient \rbrace$.
 This is where they use the order of $\GroupFin$ because they only get~$\dim(\RefQuotient_{\GroupFin}) \leq (n+1) \lvert \GroupFin \rvert$.
 
 However, if we define~$\RefQuotient$ to be the cover of~$\FMetricSpace$ consisting of all the open balls~$B_{\delta/3}(z)$ for~$\fmetricspaceelt \in \FMetricSpace$ then we can apply Proposition~\ref{Propn:EquivariantRefinements} (to~$\RefQuotient$) to get an open $\GroupFin$-refinement~$\RefQuotient_{\GroupFin} := \Refinement$ of~$\RefQuotient$ with~$\dim(\RefQuotient_{\GroupFin}) \leq \dimn$.
 
 The diameter of every element of a cover being bounded by~$\delta /3$ is a smallness property and is inherited by the refinement~$\RefQuotient_{\GroupFin}$.
 Therefore we can continue the proof given in \cite{BartelsLueckReich2008ECHG}*{Proposition~3.2} with our version of~$\RefQuotient_{\GroupFin}$, whose dimension does not depend on the order of~$\GroupFin$, and we obtain our stronger version of the proposition.
\end{proof}

Now that we have our stronger versions of \cite{BartelsLueckReich2008ECHG}*{Propositions~3.2 and~3.3} the rest of the proof of \cite{BartelsLueckReich2008ECHG}*{Proposition~4.1} works and gives us Theorem~\ref{Thm5.6}.

\section{The existence of the long and thin covers}

Under the assumptions listed in Convention~\ref{Conventions}, Theorem~\ref{Thm5.6} gives us a way to cover a cocompact subset of~$\FSlong$ but it does not give us a cover of~$\FSshort$.
The existence of such covers was formulated as \cite{BartelsLueck2010CAT0Flow}*{Definition~5.5} and is repeated below.
\begin{definition}\label{Defn:LongFCoversAtInfinity}
 The flow space~$\FS$ \emph{admits long $\Family$-covers at infinity and periodic flow lines} if there is some number $M \in \NN$ such that for every~$\gamma >0$ there is a number~$\epsilon >0$ and a collection~$\mathcal{V}$ of open $\Family$-subsets of~$\FS$ satisfying
 \begin{enumerate}
  \item $\mathcal{V}$ is $G$-invariant,
  \item $\dim \mathcal{V} \leq M$,
  \item there is a compact~$K \subseteq \FS$ with
   \begin{itemize}
    \item $\FSshort \cap GK = \emptyset$,
    \item for every $z \in \FS$ there is a set~$V \in \mathcal{V}$ with $B_{\epsilon} \big( \flow_{[-\gamma,\gamma]} (z) \big) \subseteq V$.
   \end{itemize}
 \end{enumerate}
\end{definition}

If a flow space admits long $\Family$-ocvers at infinity and periodic flow lines then we can find a cover of the entire flow space, which is made formal in the following theorem.

\begin{theorem}\label{Thm5.7}
 Under the assumptions given in Convention~\ref{Conventions}, if~$\FS$ admits long $\Family$-covers at infinity and periodic flow lines and $\Family$ contains all virtually cyclic subgroups of~$G$ then there exists a number $\hat{N} \in \NN$ such that for any~$\alpha >0$ there is some~$\epsilon >0$ and an open $\Family$-cover $\mathcal{U}$ of~$\FS$ satisfying
 \begin{enumerate}
  \item $\dim \mathcal{U} \leq \hat{N}$,
  \item for every $z \in \FS$ there is a set $U \in \mathcal{U}$ with $B_{\epsilon} \big( \flow_{[-\alpha,\alpha]} (z) \big) \subseteq U$,
  \item $\mathcal{U} / G$ is finite.
 \end{enumerate}
\end{theorem}
\begin{proof}
 This is \cite{BartelsLueck2010CAT0Flow}*{Theorem~5.7} except we use our Theorem~\ref{Thm5.6} instead of their \cite{BartelsLueck2010CAT0Flow}*{Theorem~5.6}.
\end{proof}

Let's list some consequences. Wegner introduced the condition strongly transfer reducible to show that CAT(0)-groups satisfy the Farrell-Jones conjecture (see \cite{Wegner2012FJCforCAT}*{Definition~3.1} and \cite{Wegner2012FJCforCAT}*{Theorem~1.1}).

Then, using work of \cite{BartelsLueck2010CAT0Flow}, Wegner proved that a CAT$(0)$-group is strongly transfer reducible (see \cite{Wegner2012FJCforCAT}*{Theorem~3.4}).
We can generalise this slightly to allow the group to act non-cocompactly.

\begin{theorem}\label{thm:longfimplFJ} Let $G$ be group and $X$ be a $G$-space. Assume that
\begin{enumerate}
\item $X$ is a proper CAT(0)-space,
\item $X$ has finite covering dimension,
\item The group action is proper and isometric,
\item $FS(X)$ admits long $\mathcal{F}$-covers at infinity and periodic flow lines,% (in the sense of \cite[Definition~5.5]{BartelsLueck2010CAT0Flow}).
\end{enumerate}
then $G$ is strongly transfer reducible over the family $\Vcyc\cup \mathcal{F}$. 
\end{theorem}

\begin{remark} Hence $G$ satisfies the $K$-theoretic Farrell-Jones conjecture % in algebraic $K$-theory
 relative to the family $\Vcyc \cup \mathcal{F}$ in all dimensions (by \cite{Wegner2012FJCforCAT}*{Theorem~1.1}).
\end{remark}

\begin{proof}[Proof of Theorem~\ref{thm:longfimplFJ}]
Basically Wegner's proof (\cite{Wegner2012FJCforCAT}*{Theorem~3.4}) also works here. Let us just say what the differences are.  First, he uses \cite{BartelsLueck2010CAT0Flow}*{Subsection~6.3} to show that $FS(X)$ admits long $\mathcal{F}$-covers at infinity and periodic flow lines --- here we just assume it.

Second, he uses \cite{BartelsLueck2010CAT0Flow}*{Theorem~5.7} (which needs long $\mathcal{F}$-covers as an assumption).
The condition that there is a bound on the order of finite subgroups from \cite{BartelsLueck2010CAT0Flow}*{Theorem~5.6} need not hold here, but we can use our improved version Theorem~\ref{Thm5.7} which doesn't need this condition.

All the other conditions hold as explained in \cite{BartelsLueck2010CAT0Flow}*{Section~6.2}.

These are all the neccesary modifications.
\end{proof}

For example in \cite{HenrikThesis} it is shown that the general linear group over the polynomial ring over a finite field admits long $\mathcal{F}$-covers at infinity and periodic flow lines for some family $\mathcal{F}$. These groups do not have a bound on the order of finite subgroups.

Theorem~\ref{thm:longfimplFJ} (together with \cite{Wegner2012FJCforCAT}*{Theorem~1.1}) is then used to show that they satisfy the Farrell-Jones conjecture relative to some family $\mathcal{F}$. This is a key step in showing the full conjecture for them.

One also gets using \cite{BartelsLueck2012Borel}*{Theorem~1.1 $(ii)$} that $G$ satisfies the $L$-theoretic Farrell-Jones conjecture relative to the family $\mathcal{F}_2\cup \Vcyc$, where $\mathcal{F}_2$ denotes the family of all subgroups of $G$ that have a subgroup from $\mathcal{F}$ of index at most two.
By \cite{BartelsLueckReichRueping2012GLnZ}*{Theorem~5.1} we can show that for any finite group $F$ the group $G\wr F$ satisfies the $K$- and $L$-theoretic Farrell-Jones conjecture relative to the family 
to all subgroups that are virtually a finite product of groups from $\mathcal{F}$.

\bibliographystyle{amsalpha}
\bibliography{../AdamsBib/GroupTheory}

\end{document}